\documentclass[11pt, a4paper]{amsart}
\usepackage{amsmath}
\usepackage{amssymb}
\usepackage{amsthm}
\usepackage[figuresright]{rotating}
\usepackage[all]{xy}
\usepackage{pdfsync}
\usepackage[toc,page]{appendix}
\usepackage{hyperref}
\usepackage{mathtools}
\usepackage{lipsum}
\usepackage[T1]{fontenc}        
\usepackage[utf8]{inputenc}
\usepackage[toc,page]{appendix}
\usepackage{mathtools}
\usepackage{fullpage}
\usepackage{stmaryrd}

\newcommand{\mat}[1]{\begin{bmatrix}#1\end{bmatrix}}
\newcommand{\ep}{\varepsilon}

\newcommand{\genlegendre}[4]{%
  \genfrac{(}{)}{}{#1}{#3}{#4}%
  \if\relax\detokenize{#2}\relax\else_{\!#2}\fi
}
\newcommand{\legendre}[3][]{\genlegendre{}{#1}{#2}{#3}}

\newcommand{\eq}[1]{\begin{equation}#1\end{equation}}

\theoremstyle{plain}
\newtheorem{theorem}{Theorem}[section]
\newtheorem{proposition}{Proposition}[section]
\newtheorem{corollary}{Corollary}[section]
\newtheorem{lemma}{Lemma}[section]
\newtheorem{example}{Example}[section]

\newtheorem{remark}{Remark}[section]
\newtheorem{definition}{Definition}[section]

\numberwithin{equation}{section}

\begin{document}
\author{Bora Yalkinoglu} 
\address{CNRS and IRMA, Strasbourg}
\email{yalkinoglu@math.unistra.fr}
\date{\today}
\title{A note on Shintani's invariant}

\maketitle

\begin{abstract}
\noindent
Shintani's celebrated invariants are conjectured to generate abelian extensions of real quadratic number fields, offering a potential solution to Hilbert's 12th problem in that setting. In this note, we derive new expressions for Shintani's invariants by generalizing an observation of Yamamoto, who showed that these invariants - originally formulated using the double sine function - can be expressed in terms of the q-Pochhammer symbol.   
\end{abstract}

\section{Introduction}
\noindent
In his seminal paper \cite{Shintani1977}, Shintani formulated a conjecture stating that certain invariants \eq{X(\mathfrak f) = \prod_{k=0}^{g(\mathfrak f)} \mathcal S(\ep,z_k)\mathcal S(\ep',z_k'),} defined via products of special values of the double sine function $\mathcal S(\omega,z)$, generate abelian extensions of real quadratic number fields. If true, this conjecture would provide a solution to Hilbert's 12th problem in the real quadratic setting. However, the conjecture remains completely open, due to the absence of a geometric framework analogous to that of elliptic curves with complex multiplication, which plays a central role in the classical theory for imaginary quadratic number fields. \\ \\
Despite extensive and convincing numerical evidence supporting Shintani’s conjecture, the double sine function remains mysterious. To date, no non-trivial theoretical explanation has been found for the validity of Shintani's conjecture. \\ \\
The goal of this note is to generalize a key observation by Yamamoto \cite{yamamoto2010factorization}, who showed that Shintani’s invariants $X(\mathfrak f)$ can be expressed as a (limit of a) quotient of $q$-Pochhammer symbols. Specifically, we prove the following (Theorem \ref{maintheorem}):
\begin{theorem}
Let $K=\mathbb Q(\sqrt{d})$ be a real quadratic number field with positive fundamental unit $\ep = \frac{a+b\sqrt{d}}{2} \in \mathcal O_{K,+}^\times$, with $a,b \in \mathbb N$, such that $\langle 1,\varepsilon \rangle_\mathbb Z = \mathcal O _K$. Let $\mathfrak f = (u+v\sqrt{d}) \in I_K$ be a principal ideal. Then there exist $g = g(\mathfrak f) \in \mathbb N$ and $(x,y) \in \mathbb Q ^2$, such that Shintani's invariant $X(\mathfrak f)$ is given by \eq{\label{expression}X(\mathfrak f) = \lim_{n\to \infty}\left\vert \frac{(x,y;\tau_{n-g})_\infty}{(x,y;\tau_{n+g})_\infty} \frac{(x,y;\tau_{-n-g})_\infty}{(x,y;\tau_{-n+g})_\infty} \right\vert,}
where \eq{(x,y;\tau)_\infty = \prod_{k\geq 0}(1-e^{2 \pi i (x\tau+y) e^{2 \pi i k \tau}})} is the $q$-Pochhammer symbol and $\tau_n = \frac{T_{n+1}(a)+i b \sqrt{d}}{T_n(a)}$.
\end{theorem}

\noindent The proof relies on a close link between the arithmetic of real quadratic fields and (minus) continued fractions \cite{zagier2013zetafunktionen,yamamoto2008kronecker}, which provides a natural description of the decomposition data $\{z_k\}_{k\in \{1,..,g\}}$ appearing in $X(\mathfrak f)$. We also utilize an important product formula for the double sine function due to Shintani (Theorem \ref{shintaniproduct}), which enables telescoping expressions for Shintani's invariants. \\
Several corollaries (\ref{corollary1} - \ref{corollary3}) provide alternative formulations and refinements. \\
It is worth noting that the discrete parameter $\tau_n$, defined using classical Chebyshev polynomials $T_n(x)$, appears to be novel in this context. \\ \\
In another direction, recent breakthroughs on Hilbert's 12th problem from a $p$-adic perspective - see \cite{dasgupta2021brumer,darmon2021singular} - highlight the importance of connecting Shintani's archimedean viewpoint (via the double sine function or $q$-Pochhammer symbols) with these developments. Understanding this connection remains a significant challenge. Let us mention a recent cohomological interpretation of Shintani invariants due to Kopp, see \cite{kopp}, as well as the possibly related works \cite{bergeron2023elliptic,gendronI}.
\\ \\
We aim to present our results in as accessible a manner as possible, prioritizing clarity over generality to highlight the key ideas. In forthcoming work, we will further explore these invariants using the cyclic quantum dilogarithm.

\subsubsection*{Summary}
In section 2, we recall Shintani's definition of his invariants using the formalism of minus continued fractions (see \cite{zagier2013zetafunktionen,yamamoto2008kronecker}), which is used to define the decomposition data. In section 3, after recalling important results due to Shintani and Yamamoto, we introduce a discretization of certain modular geodesics and use this to prove our main theorem expressing Shintani's invariants in terms of $q$-Pochhammer symbols. In particular, we explain how one can express Shintani's invariants using a single $q$ parameter. In the appendix, we provide background on Chebyshev polynomials and prove a result concerning the size of the quantity $g(p)$, for $p$ a prime number.

\subsubsection*{Acknowledgements} We would like to thank Giuseppe Ancona for his helpful feedback on the manuscript, and Thomas Dreyfus for interesting discussions on $q$-difference equations. We also wish to express our sincere gratitude to the (anonymous) referees for carefully reading the manuscript and for their valuable comments, which considerably improved this note.
\subsection{Notations}
\noindent
Let $d$ be a positive, square-free integer and $K=\mathbb Q(\sqrt d)$ be a real quadratic number field. We denote by $\mathcal O _K$ the ring of integers of $K$, by $\mathcal O _K ^\times$ the group of units and by $\mathcal O _{K,+}^\times$ the group of totally positive units. The monoid of non-zero integral ideals of $\mathcal O _K$ is denoted by $I_K$. For any $x \in K$, we write $x'$ for its Galois conjugate under the non-trivial automorphism $\sigma \in \mathrm{Gal(K/\mathbb Q)}$. \\ \\
We fix a totally positive fundamental unit $\ep = \frac{a+b\sqrt{d}} 2 \in \mathcal O _{K,+} ^\times$ with $a,b \in \mathbb N$. For an ideal $\mathfrak f \in I_K$, we define $g(\mathfrak f) \in \mathbb N$ to be the smallest positive integer such that \eq{\langle \ep^{g(\mathfrak f)} \rangle = \mathcal O_{K,+}^\times \cap (1+\mathfrak f) ,} and we write $\ep_\mathfrak f = \ep^{g(\mathfrak f)}$. \\ \\
For $r \in \mathbb R$, we define $\langle r \rangle \in \mathbb R$ by $r - \langle r \rangle \in \mathbb Z$ and $0 < \langle r \rangle \leq 1$. In the same way, we define $[r] \in \mathbb R$ by $r - [r] \in \mathbb Z$ and $0\leq [r] < 1$.
 
 \section{Shintani's invariant}
 
 \subsection{Shintani's invariant}
We follow the exposition in \cite{Shintani1977,yamamoto2008kronecker}.
Let $\mathfrak C \in \mathrm{Cl}_K(\mathfrak f)$ be a (strict) ray class of conductor $\mathfrak f \in I_K$. Shintani defined the associated partial zeta function as \eq{\zeta(s,\mathfrak C) = \sum_{\mathfrak a \in \mathfrak C, \mathfrak a \text{ integral}} \mathfrak{N}(\mathfrak a)^{-s}} and introduced the invariant \eq{X(\mathfrak C) = \mathrm{exp}(-\zeta'(0,\mathfrak C) + \zeta'(0,\mathfrak C ')).}
These invariants $X(\mathfrak C)$, known as the Shintani-Stark units, are conjectured to generate (essentially) the maximal abelian extension $K^{ab}$ of $K$. Shintani further expressed these invariants in terms of the double sine function: \eq{\label{doublesine}\mathcal S(\omega,z) = \mathrm{exp}(\zeta_2'(0,1+\omega-z,\omega)-\zeta'_2(0,z,\omega)),}
where the Barnes double zeta function $\zeta_2(s,z,(1,\omega))$ is given by \eq{\zeta_2(s,z,\omega) = \sum_{n_1,n_2=0}^\infty (n_1+n_2 \omega+z)^{-s}.}
\noindent Later, we will need the following basic symmetry (see Proposition 3.3.1 \cite{yamamoto2008kronecker})
\begin{lemma}
\label{symmetry}
The double sine function satisfies \eq{\mathcal S(\omega,z) = 2 \, \mathrm{sin}(\pi z) \, \mathcal S(\omega,z+\omega).}
\end{lemma}
\noindent Shintani proved the following representation:
\eq{X(\mathfrak C) = \prod_k \mathcal S(\ep_k,z_k) \,\mathcal S(\ep_k',z_k'),}
where $\mathcal Z _{\mathfrak C} = \{(\ep_k,z_k)\}$ is a finite set known as the decomposition datum associated to the class $\mathfrak C$.
Following \cite{yamamoto2008kronecker}, we define \eq{\label{ShintaniValues}X_1(\mathfrak C) = \prod_{k}\mathcal S(\ep_k,z_k) \text{ and } X_2(\mathfrak C) = \prod_k \mathcal S(\ep_k',z_k').}
\noindent To simplify the exposition, we restrict to the case where $\mathfrak C$ is the trivial class $1_\mathfrak f \in \mathrm{Cl}_K(\mathfrak f)$. We write $X(\mathfrak f) = X(1_\mathfrak f)$ and $X_i(\mathfrak f) = X_{i}(1_\mathfrak f) $ for $i \in \{1,2\}$. \\ \\
Further, we assume that $\mathfrak f \in I_K$ is a principal ideal and more importantly, we assume that \eq{\langle 1,\ep\rangle_\mathbb Z = \mathcal O _K.} \begin{lemma} There are infinitely many real quadratic number fields $K$, such that $\langle 1,\ep\rangle_\mathbb Z = \mathcal O _K$. \end{lemma}
\begin{proof}
One can take the family $K_m = \mathbb Q(\sqrt{4m^2-1})$, for $m \geq 1$. In this case, $\ep_m = 2m+\sqrt{4m^2-1}$ is a positive fundamental unit which also generates $\mathcal O_{K_m}$.
\end{proof}
\begin{remark}
There are more real quadratic number fields $K=\mathbb Q(\sqrt{d})$ with the property $\langle 1,\ep\rangle_\mathbb Z = \mathcal O_K$, e.g., the property holds for $d =5$ and $d=21$.
\end{remark}
\noindent Under these assumptions, the partial zeta function admits the representation \eq{\zeta(s,1_\mathfrak f) = \sum_{\beta \in (\mu+\mathfrak b)_+ / \langle \ep_\mathfrak f \rangle} \mathfrak N (\beta)^{-s},} with $\mu \in K^\times$ such that \eq{\mathfrak b = \langle 1,\ep\rangle_\mathbb Z = (\mu) \, \mathfrak f.}
\begin{remark}
For $\mathfrak f = (u+v\sqrt{d}) \in I_K$ a principal ideal, we can take \eq{\mu = \tfrac{1}{u+v\sqrt{d}}=\tfrac{u-v\sqrt{d}}{\mathfrak N (\mathfrak f)}.} In particular, if $v=0$, we have $\mu = \frac 1 u$.
\end{remark}
\begin{lemma}
For every real quadratic number field $K$, with our choice of a positive fundamental unit $\ep = \tfrac{a+b\sqrt{d}}{2} \in \mathcal O_{K,+}^\times$, the minus continued fraction expansion of $\ep$ (see \cite{zagier2013zetafunktionen}) has length one, i.e., we have \eq{\ep = \llbracket a \rrbracket=
a-\cfrac{1}{a-\cfrac{1}{a-
  \cdots\vphantom{\cfrac{1}{1}} }}.}
\end{lemma}
\begin{proof}
From our assumption we know $\ep > 1$ and together with $\ep + \tfrac1 \ep = a$, the result follows.
\end{proof}
\noindent In our set-up, Shintani's famous cone decomposition theorem now looks as follows (cf., \cite{yamamoto2008kronecker}):
\begin{proposition}
\label{decomposition}
Under the above assumptions, for each $k \in \mathbb Z$, there exists a unique pair $(x_k,y_k) \in \mathbb Q^2$ with $0<x_k\leq1$,  $0 \leq y_k<1$, such that \eq{x_k \ep^{1-k} + y_k \ep^{-k} \in \mu+\mathfrak b.}
In particular, \eq{(\mu+\mathfrak b)_+ / \langle \ep_\mathfrak f \rangle = \coprod_{k=1}^{g(\mathfrak f)}\{ (x_k+u)\ep^{1-k}+(y_k+v)\ep^{-k} | u,v \in \mathbb Z_{\geq 0}\}.}
\end{proposition} 
\noindent This gives the decomposition datum \eq{\mathcal Z_\mathfrak f = \{(\ep,x_k \ep + y_k) \, \vert \, 0\leq k \leq g(\mathfrak f)-1\}} of the class $1_\mathfrak f$, which leads to \eq{ \label{ShintaniValues2}X_1(\mathfrak f) = \prod_{k=1}^{g(\mathfrak f)} \mathcal S(\ep,x_k \ep + y_k) \text{ and } X_2(\mathfrak f) = \prod_{k=1}^{g(\mathfrak f)} \mathcal S(\ep',x_k \ep' + y_k).}
\begin{remark}
We sometimes abbreviate the decomposition datum as \eq{\mathcal Z_\mathfrak f = \{(x_k,y_k) | k\in \{0,\dots,g(\mathfrak f)-1\}\}.}
\end{remark}
 \begin{remark}
For an arbitrary real quadratic number field $K$, Lemma \ref{epsilonf} shows that for any prime $p \in \mathbb N$ not dividing the discriminant $\Delta$ of $K$, one always has \eq{g(p) \, \vert \, p-  \legendre{\Delta}{p}.}
\end{remark}
\begin{remark}
Define $(\tilde x _k, \tilde y _k) = ([x_k],y_k)$ and set $\tilde z_k = \tilde x _k \ep + y_k$. Then, using Lemma \ref{symmetry}, the slightly modified invariant \eq{\tilde X(\mathfrak f) = \sideset{}{'}\prod_{k=1}^{g(\mathfrak f)} \mathcal S(\ep,\tilde z _k) \, \mathcal S (\ep',\tilde z_k'),} where we ignore the factors with $\tilde z_k = \tilde z_k'=0$, satisfies 
\eq{X(\mathfrak f) \in K^{\mathrm{ab}} \Leftrightarrow \tilde X(\mathfrak f) \in K^{\mathrm{ab}} .}
This comes from $\mathcal S(\ep,0) = 0$ and $\mathcal S(\ep,\ep) = \ep^{-1/2} \in K^{\mathrm{ab}}$, cf., \cite{yamamoto2008kronecker}.
\end{remark}
 \subsection{Explicit decomposition data for principal ideals}
The matrix \eq{U = \begin{bmatrix} a & -1 \\ 1 & 0 \end{bmatrix} \in \mathrm{SL}_2(\mathbb Z)}
encodes the minus continued fraction expansion of $\ep$, since \eq{U \cdot \ep = a-\frac 1 \ep = \ep,} where the action is by Möbius transformation. This implies that the eigenvalues of $U$ are precisely $\{\ep,\ep'\}$. 
Following \cite{yamamoto2008kronecker}, for $k \in \mathbb Z$ and an arbitrary $\mathfrak f \in I_K$, the decomposition data satisfy the recurrence: \eq{\label{xyrecursion}(x_{k+1},y_{k+1}) = (\langle a x_k + y_k \rangle,[-x_k]) = (\langle a x_k -x_{k-1} \rangle,1-x_k),} which corresponds to the matrix relation (modulo $1$) \eq{\begin{bmatrix}x_{k+1} \\ y_{k+1} \end{bmatrix} = U^{T} \begin{bmatrix}x_k \\ y_k\end{bmatrix} = \begin{bmatrix}a & 1 \\ -1 & 0\end{bmatrix} \begin{bmatrix}x_k \\ y_k\end{bmatrix} .}
From Proposition \ref{decomposition}, we thus obtain 
\begin{lemma}
\label{LemmaPeriodic}
For every $k \in \mathbb Z$, we have (upon properly taking modulo $1$) \eq{(U^T)^{g(\mathfrak f)} \begin{bmatrix}x_k \\ y_k\end{bmatrix} = \begin{bmatrix}x_k \\ y_k\end{bmatrix}.} In particular, \eq{(x_0,y_0) = (x_{g(\mathfrak f)},y_{g(\mathfrak f)}).}
\end{lemma} 
\noindent  
To make this more explicit, let $\mathfrak f = (u + v \sqrt{d}) \in I_K$ be a principal ideal. We set \eq{\mu = \frac{u-v\sqrt{d}}{\mathfrak N(\mathfrak f)}} and write $\mu = x\ep + y$ with $x,y \in \mathbb Q$. A straightforward calculation yields: \eq{\label{denominators}x = -\tfrac{2 v}{b \, \mathfrak N(\mathfrak f)} \text{   and   } y = \tfrac{bu+av}{b \, \mathfrak N(\mathfrak f)}.} Thus, the initial pair $(x_0,y_0)$ for the cone decomposition is given by \eq{\label{initialu0}(x_0,y_0) = (\langle x \rangle,[y]).} For $k \geq 0$, the identity \eq{(U^T)^k = \mat{U_k(a)&U_{k-1}(a) \\ - U_{k-1}(a) & -U_{k-2}(a)},} (see formula (\ref{Urecursion})) leads to the explicit formula \eq{\label{generalformula}(x_k,y_k) = (\langle U_k(a) x_0 + U_{k-1}(a) y_0 \rangle,[-U_{k-1}(a) x_0 - U_{k-2}(a) y_0]).}
If we specialise further to $v=0$, we find \eq{(x_k,y_k) = (\langle \tfrac{U_{k-1}(a)}{u} \rangle,[-\tfrac{U_{k-2}(a)}{u}]),} so we only need to understand the behaviour of $U_{k}(a)$ modulo $u$. 
In this case, we have: \eq{\label{initialu}(x_0,y_0) = (1,\frac{1}{u}).}

\subsection{Examples}
\label{examples}
We follow Shintani's original examples from \cite{Shintani1977}.

\begin{example}
For $K = \mathbb Q(\sqrt 5 )$, we have $\ep = \tfrac{3+\sqrt{5}} 2$. For the ideal $\mathfrak f = (4)$, we have $g = 3$ and the decomposition data \eq{\mathcal Z _\mathfrak f = \{(1,\tfrac 1 4),(\tfrac 1 4,0),(\tfrac 3 4 , \tfrac 3 4)\}.} Thus, \eq{X_1(4) = \mathcal S(\ep,\ep+\tfrac 1 4)\mathcal S (\ep,\tfrac \ep 4)\mathcal S(\ep,\tfrac{3 \ep}{4}+\tfrac 3 4),} \eq{X_2(4) = \mathcal S(\ep',\ep'+\tfrac 1 4)\mathcal S (\ep',\tfrac{\ep'} 4)\mathcal S(\ep',\tfrac{3 \ep'} 4 + \tfrac 3 4).} Shintani was able to explicitly compute \eq{X(4) = \left(\tfrac{1+\sqrt{5}}{2}-\sqrt{\tfrac{1+\sqrt{5}}{2}}\right ).} 
\end{example}
 
\begin{example}
Again, for $K=\mathbb Q(\sqrt 5)$, take $\mathfrak f = (4-\sqrt{5})$. Then $g = 5$ and \eq{\mathcal Z _\mathfrak f = \{(\tfrac 2 {11},\tfrac 1 {11}),(\tfrac 7 {11},\tfrac 9 {11}),(\tfrac 8 {11} , \tfrac 4 {11}),(\tfrac 6 {11}, \tfrac 3 {11}),(\tfrac {10} {11}, \tfrac 5 {11})\}.} 
Using different techniques (not based on the double sine function), Shintani showed: \eq{X(4-\sqrt{5}) = \frac 1 2 \left (\tfrac {3+\sqrt{5}} 2 - \sqrt{\tfrac{3 \sqrt{5}-1} 2} \right).}
\end{example}
\begin{example}

For $K = \mathbb Q(\sqrt{21})$, we have $\ep = \tfrac{5+\sqrt{21}}{2}$. For $\mathfrak f = (3)$, we find $g = 3$ and \eq{\mathcal Z _\mathfrak f = \{(1,\tfrac 1 3),(\tfrac 1 3 ,0),(\tfrac 2 3 , \tfrac 2 3)\}.} Shintani was able to evaluate \eq{X(3) = \frac 1 2 \left (\tfrac{1+\sqrt{21}}{2}-\sqrt{\tfrac{3+\sqrt{21}}{2}} \right ).}
\end{example}

\begin{remark}
It is important to note that Shintani’s explicit evaluations in the above examples did not rely on properties of the double sine function $\mathcal S(\omega,z)$! 
\end{remark}

\section{The $q$-Pochhammer function and Shintani's invariant}

\subsection{Results of Shintani and Yamamoto}
A lesser-known result due to Shintani \cite{Shintani1977} is the following product formula for the double sine function:
\begin{theorem}
\label{shintaniproduct}
For $\mathrm{Im}(\tau) > 0$, we have \eq{
\mathcal S(\tau,z) = i^{1/2} e^{\frac {\pi i}{12} (\tau + \frac 1 \tau)} e^{\frac{\pi i}{2} (\frac {z^2}\tau - (1+\frac 1 \tau)z )}\tfrac{\prod_{m \geq 0} (1-e^{2 \pi i (m \tau + z)}) }{\prod_{m \geq 1} (1-e^{2 \pi i (\frac{-m+z} \tau )}) }.}
\end{theorem}
\noindent The special function appearing in Shintani's theorem is well-known:
\begin{definition}
The (infinite) $q$-Pochhammer symbol is defined by \eq{(x,y;\tau)_\infty = \prod_{k \geq 0} (1-e^{2 \pi i (k \tau + x\tau+y)})} or, equivalently, in terms of $q = e^{2 \pi i \tau}$ and $\zeta = e^{2 \pi i (x\tau + y)}$, as \eq{ (\zeta;q)_\infty = \prod_{k \geq 0}(1-\zeta q^k).}
\end{definition}
\begin{lemma}
\label{symmetries}
The $q$-Pochhammer symbol satisfies the following symmetries: 
\begin{itemize}
\item[1)] $(x,y;\tau +1)_\infty = (x,x+y;\tau)_\infty$,
\item[2)] $(x,y+1;\tau)_\infty = (x,y;\tau)_\infty$.
\end{itemize}
\end{lemma}
\noindent Later, Yamamoto \cite{yamamoto2010factorization} observed:
\begin{proposition}
\label{YamamotoProposition}
For $\ep > 0$, $x,y \in \mathbb R_{\geq 0}$ and $\mathrm{Im}(\tau) > 0$, we have \eq{\mathcal S(\ep,x\ep + y) = \lim_{\tau \to \ep}|\mathcal S(\tau,x\tau + y)| = \lim_{\tau \to \ep} \left \vert \tfrac{(x,y;\tau)_\infty} { (1-y,x;-\frac 1 \tau)_\infty} \right \vert .}
\end{proposition}
\begin{proof}
This follows from the definition (cf., (\ref{doublesine})) and analyticity properties of the double sine function $\mathcal S(\omega,z)$ (see \cite{Shintani1977,ruijsenaars}) .
\end{proof}

\subsection{Modular geodesics}
To make the limit $\tau \to \ep$ from Proposition \ref{YamamotoProposition} precise, we consider the modular geodesic connecting $\ep$ and its Galois conjugate $\ep ' = \ep ^{-1}$ in the upper half-plane $\mathbb H$. Define the matrix \eq{A = \begin{bmatrix} 1 & 1 \\ \ep ' & \ep \end{bmatrix},} which diagonalizes $U$, i.e., \eq{A^{-1}U A = \begin{bmatrix} \ep & 0 \\ 0 & \ep '\end{bmatrix}.} Then, the family \eq{\tau_t = U A \begin{bmatrix}e^{t/2} & 0 \\ 0 & e^{-t/2}\end{bmatrix} \cdot i} describes the modular geodesic from $\ep$ to $\ep '$, satisfying: \eq{\mathrm{Im}(\tau_t)>0 \text{ for all } t \in \mathbb R \text{ and }  \lim _{t\to \pm \infty} \tau_t = \ep^{\pm 1}.}
Let us now discretize the variable $t$ by setting, for $n \in \mathbb Z$, \eq{t_n = \mathrm{ln}(\ep^{n}).} This leads to 
\begin{lemma}
The sequence \eq{\tau_n = \tau_{t_n} = \tfrac{T_{n+1}(a) + i b \sqrt{d}}{T_n(a)},} for $n \in \mathbb Z$, gives a discrete approximation along the modular geodesic connecting $\ep$ and $\ep'$, where $T_n(x)$ are the Chebyshev polynomials of the first kind.
\end{lemma}
\begin{proof}
We have $UA = \mat{\ep & \ep' \\ 1 & 1}$ and $\mat{\ep^{n/2}&0\\0&\ep^{-n/2}} \cdot i = i \ep^n$, which leads to \eq{\tau_n = \mat{ \ep & \ep' \\ 1 & 1} \cdot i \ep^n = \tfrac{i\ep^{n+1}+\ep'}{i\ep^n+1} = \tfrac{(1-i\ep^n)(i\ep^{n+1}+\ep^{-1})}{\ep^n T_n(a)}=\tfrac{T_{n+1}(a)+i(\ep-\ep')}{T_n(a)}.}

\end{proof}

\begin{lemma}\label{LemmaAction}For all $k \in \mathbb Z$, we have the shift relation
\eq{U^{k/2} \cdot \tau_n = \tau_{n+k}.} 
\end{lemma}
\begin{proof}
We have $\mat{\ep^{1/2}&0 \\ 0 & \ep^{-1/2}} = A^{-1}U^{1/2}A$, which implies $U^{k/2}A = A \mat{\ep^{k/2}&0\\0&\ep^{-k/2}}$. Thus \eq{U^{k/2} \cdot \tau_n = U^{k/2} U A \mat{\ep^{n/2}&0\\0&\ep^{-n/2}}\cdot i = UA\mat{\ep^{(n+k)/2}&0\\0&\ep^{-(n+k)/2}}\cdot i = \tau_{n+k}.}
\end{proof}
\begin{remark}Note that $\lim _{n \to \pm \infty} \tfrac{T_{n+k}(a)}{T_n(a)} = \ep ^{\pm k}.$
\end{remark}
\subsection{A new formula for Shintani's invariants}
We now state our main result, which generalizes Yamamoto’s observation \cite{yamamoto2010factorization}.
\begin{theorem}
\label{maintheorem}
Under the previous assumptions, let $\mathfrak f = (u+v\sqrt{d}) \in I_K$ be a principal ideal, with $g = g(\mathfrak f)$, and let  $(x_0,y_0) = (\langle x_{\mathfrak f}\rangle,[y_{\mathfrak f}])$ be the initial data (cf., (\ref{initialu0})). Then: \eq{\label{expression0}X_1(\mathfrak f) = \lim_{n\to \infty}\left\vert \tfrac{(x_0,y_0;\tau_{n})_\infty}{(x_0,y_0;\tau_{n+2g})_\infty} \right\vert \text{ and } X_2(\mathfrak f) = \lim_{n\to \infty}\left\vert \tfrac{(x_0,y_0;\tau_{-n})_\infty}{(x_0,y_0;\tau_{-n+2g})_\infty} \right\vert.}
\end{theorem}
\begin{proof}

The idea is to exploit Proposition \ref{YamamotoProposition}, in order to obtain a description of the form \eq{\mathcal S(\ep,x_k\ep+y_k)=\lim_{n\to\infty}\left\vert \tfrac{(x_k,y_k;\tilde\tau_k)_\infty}{(x_{k-1},y_{k-1};\tilde\tau_{k-1})_\infty}\right\vert.} For this, we define \eq{\tilde \tau_k = \tilde \tau _{n,k} =  U^{g-k}\cdot \tau_{n}.} Then, we have \eq{\tilde \tau_{k-1} = U \cdot \tilde \tau_{k} = a - \frac {1} {\tilde \tau_k} \text{ \ \ \   and \ \ \    }\lim_{n\to \infty}\tilde \tau_k = \ep,} which means (using Proposition \ref{YamamotoProposition}) \eq{\mathcal S(\ep,x_k\ep+y_k) = \lim_{n\to \infty} \left \vert \tfrac{(x_k,y_k;\tilde \tau_k)_\infty}{(1-y_k,x_k;-\tfrac{1}{\tilde \tau _k})_\infty} \right \vert .} Using Lemma \ref{symmetries} and the recursion $(x_k,y_k) = (\langle a x_{k-1}+y_{k-1}\rangle,1-x_{k-1})$ (see (\ref{xyrecursion})), we obtain the following key observation for the denominator \begin{align}(1-y_k,x_k;-\tfrac{1}{\tilde \tau _k})_\infty &= (1-y_k,x_k-a(1-y_k);a-\tfrac{1}{\tilde \tau_k})_\infty \\ &= (x_{k-1},\langle a x_{k-1}+y_{k-1}\rangle-a x_{k-1};\tilde\tau_{k-1})_\infty \\&=(x_{k-1},y_{k-1};\tilde \tau _{k-1})_\infty.\end{align}
Using $(x_0,y_0)=(x_g,y_g)$, $U^g \cdot \tilde\tau_g = \tilde \tau_0$ and $\tilde\tau_g = \tau_n$, we get \begin{align}X_1(\mathfrak f) &= \prod_{k=1}^{g} \mathcal S(\ep,x_k\ep+y_k) = \lim_{n\to \infty}\left \vert\prod_{k=1}^{g} \tfrac{(x_k,y_k;\tilde\tau_k)_\infty}{(x_{k-1},y_{k-1};\tilde\tau_{k-1})_\infty} \right \vert \\ &= \lim_{n\to\infty}\left\vert \tfrac{(x_g,y_g;\tilde\tau_g)_\infty}{(x_0,y_0;\tilde\tau_0)_\infty} \right\vert = \lim_{n\to\infty}\left \vert \tfrac{(x_g,y_g;\tilde\tau_g)_\infty}{(x_g,y_g;U^{g}\cdot \tilde\tau_g)_\infty}\right\vert \\ &= \lim_{n\to \infty}\left\vert \tfrac{(x_0,y_0;\tau_n)_\infty}{(x_0,y_0;\tau_{n+2g})_\infty} \right\vert.\end{align}
Moreover, setting \eq{\tilde\tau_k' = \tilde\tau_{n,k}' = U^{g-k}\cdot \tau_{-n},} we have \eq{\tilde\tau_{k-1}' = U \cdot \tilde \tau_k' \ \ \  \text{ and } \ \ \  \lim_{n\to \infty}\tilde\tau_k' = \ep',} showing with same argument (verbatim) \eq{X_2(\mathfrak f) = \lim_{n\to \infty}\left\vert \tfrac{(x_0,y_0;\tau_{-n})_\infty}{(x_0,y_0;\tau_{-n+2g})_\infty} \right\vert.}
\end{proof}
\begin{remark}
A general version of this theorem, for general $\mathfrak f$ and $\mathfrak C$, is available. For clarity, we only consider the case where $\mathfrak f$ is a principal ideal. 
\end{remark}

\begin{corollary}
\label{corollary1}
 Under the previous assumptions, for all $k \in \{0,\dots,g-1\}$, the invariant satisfies \eq{\label{formula2}X(\mathfrak f) = \lim_{n\to\infty} \left\vert \tfrac{(x_k,y_k;\tau_n)_\infty(x_k,y_k;\tau_{-n})_\infty}{(x_k,y_k;\tau_{n+2g})_\infty(x_k,y_k;\tau_{-n+2g})_\infty} \right\vert=\lim_{n\to \infty}\left\vert \tfrac{(x_k,y_k;\tau_{n-g})_\infty(x_k,y_k;\tau_{-n-g})_\infty}{(x_k,y_k;\tau_{n+g})_\infty(x_k,y_k;\tau_{-n+g})_\infty} \right\vert.}
\end{corollary}
\begin{proof}
The first equality simply comes from the shift invariance $(x_k,y_k)\mapsto (x_{k+1},y_{k+1})$ of the Shintani invariant $X(\mathfrak f)$. The second equality comes from the shift $n\mapsto n-g$.
\end{proof}
\begin{corollary}
\label{corollary2}
Under the the previous assumptions, if $\mathfrak f = (u) \in I_K$, and we define $q_n = e^{2\pi i \tau_n}$, $\zeta_u = e^{2 \pi i / u}$, then: \eq{\label{formula3}X(\mathfrak f) = \lim_{n\to\infty}\left\vert\tfrac{(\zeta_u;q_n)_\infty(\zeta_u;q_{-n})_\infty}{(\zeta_u;q_{n+2g})_\infty(\zeta_u;q_{-n+2g})_\infty}\right\vert= \lim_{n\to\infty}\left\vert\tfrac{(\zeta_u;q_{n-g})_\infty(\zeta_u;q_{-n-g})_\infty}{(\zeta_u;q_{n+g})_\infty(\zeta_u;q_{-n+g})_\infty}\right\vert.}
\end{corollary}
\begin{proof}
This follows from $(x_0,y_0) = (1,\tfrac 1 u)$, see (\ref{initialu}).
\end{proof}
\noindent
The formulas (\ref{expression0}), (\ref{formula2}) and (\ref{formula3}) contain two different $q$ parameters. We want to show that we can express $X(\mathfrak f)$ in terms of a single $q$ parameter:
\begin{corollary}
\label{corollary3}
Continuing with the same assumptions, let $\mathfrak q = e^{- 2 \pi \sqrt{d}}$, $\zeta_0 = e^{2 \pi i (x_0\ep+y_0)}$, $T_n = T_n(a)$  and \eq{\zeta_{n,r} = \zeta_0 e^{2 \pi i r\frac { T_{n+1}}{T_n}  }e^{-2\pi b \sqrt{d}\frac{r}{T_n}},} for $0\leq r < T_n$. Then, we have
\eq{\label{expression2}X_1(\mathfrak f) = \lim_{n\to \infty} \left \vert \tfrac{\prod_{r=0}^{T_{n-g}} (\zeta_{n-g,r};\mathfrak q)_\infty} {\prod_{r=0}^{T_{n+g}} (\zeta_{n+g,r};\mathfrak q)_\infty} \right \vert \text{ and } X_2(\mathfrak f) = \lim_{n\to \infty} \left \vert \tfrac{\prod_{r=0}^{T_{n+g}} (\zeta_{-n-g,r};\mathfrak q)_\infty} {\prod_{r=0}^{T_{n-g}} (\zeta_{-n+g,r};\mathfrak q)_\infty} \right \vert .}

\end{corollary}
\begin{proof}
This follows from the multiplicative dependence of $q_{n-g} = e^{2 \pi i \tau_{n-g}}$ and $q_{n+g}= e^{2 \pi i \tau_{n+g}}$, as  \eq{ \mathfrak q = q_{n-g}^{T_{n-g}} = q_{n+g}^{T_{n+g}}.} From \eq{ (\zeta_0;q_{n})_\infty =  \prod_{k\geq 0}(1-\zeta_0q_{n}^k) = \prod_{s\geq 0}\prod_{r=0}^{T_{n}-1} (1-\zeta_0 q_n^{sT_{n}}q_{n}^{r}) = \prod_{s\geq0} \prod_{r=0}^{T_{n}-1}(1-\mathfrak q^s \zeta_{n,r}) = \prod_{r=0}^{T_{n} -1}(\zeta_{n,r};\mathfrak q)_\infty,} where $k = s T_n + r$, with $0\leq r < T_{n}$, the result follows. 
\end{proof}
\subsection{Examples}
The examples from section \ref{examples} can now be rewritten:
\begin{example}
For $K = \mathbb Q(\sqrt{5})$ and $\mathfrak f = (4)$, we have $a=3$, $g=3$ and 
\eq{X(4) = \lim_{n\to \infty} \left \vert \tfrac{(\zeta_4;q_{n-3})_\infty (\zeta_4;q_{-n-3})_\infty}{(\zeta_4;q_{n+3})_\infty (\zeta_4;q_{-n+3})_\infty} \right\vert .}
\end{example}
\begin{example}
Again, for $K = \mathbb Q(\sqrt{5})$, $\mathfrak f = (4-\sqrt{5})$, $g=5$ and using $(x_0,y_0) = (2/11,1/11)$, we get: 
\eq{X(4-\sqrt{5}) = \lim_{n\to \infty} \left \vert \tfrac{(2/11,1/11;q_{n-5})_\infty (2/11,1/11;q_{-n-5})_\infty}{(2/11,1/11;q_{n+5})_\infty (2/11,1/11;q_{-n+5})_\infty} \right\vert .}
\end{example}
\begin{example}
For $K = \mathbb Q(\sqrt{21})$, $\mathfrak f = (3)$, $a=5$ and $g=3$, we obtain: 
\eq{X(3) = \lim_{n\to \infty} \left \vert \tfrac{(\zeta_3;q_{n-3})_\infty (\zeta_3;q_{-n-3})_\infty}{(\zeta_3;q_{n+3})_\infty (\zeta_3;q_{-n+3})_\infty} \right\vert .}
As expected, this is in agreement with Yamamoto's original calculation of $X(3)$ in \cite{yamamoto2010factorization}.
\end{example}
\begin{remark}
It remains a highly nontrivial problem to recover Shintani’s explicit evaluations of the above examples purely from the properties of the $q$-Pochhammer symbol.
\end{remark}
\appendix
\section{Chebyshev polynomials of the first and second kind}
\noindent 
For $n \in \mathbb Z$, the (normalized) $n$-th Chebyshev polynomial of the first kind $T_n(x) \in \mathbb Z[x]$ is defined by the identity \eq{T_n(x+x^{-1}) = x^n+x^{-n}.} Similarly, the (normalized) $n$-th Chebyshev polynomial of the second kind  $U_n(x) \in \mathbb Z[x]$ is defined via: \eq{U_{n-1}(x+x^{-1}) = \tfrac{x^n-x^{-n}}{x-x^{-1}}.}
In particular, we see \eq{T_n(a) = T_n(\ep+\ep^{-1})=\ep^n +\ep^{-n} \in \mathbb N,} \eq{U_n(a) = U_n(\ep+\ep^{-1}) =\tfrac{\ep^{n}-\ep^{-n}}{\ep-\ep^{-1}}\in \mathbb N.}
These polynomials satisfy the standard identities:
\begin{itemize}
\item Multiplicative identities:
\eq{\label{ChebRelations}T_n(x)T_m(x) = T_{n+m}(x) + T_{n-m}(x), \forall n,m \in \mathbb Z , }
\eq{\label{Urecursion}U_n(x)U_m(x) = \sum_{j=0}^m U_{n-m+2j}(x), \forall n,m \in \mathbb Z .} 
\item Recurrence relations:
\eq{T_{n+1}(x) = x T_n(x) - T_{n-1}(x), \forall n \in \mathbb Z,}
\eq{U_{n+1}(x) = x U_{n}(x)-U_{n-1}(x), \forall n\in \mathbb Z.}
\item Symmetry:
\eq{T_n(x) = T_{|n|}(x), \forall n\in \mathbb Z,}
\eq{U_{-1 + n}(x) = -U_{-1-n}(x), \forall n\in \mathbb Z_{\geq 0}.}
\end{itemize}
\begin{example}
We have
\begin{itemize}
\item[] $T_{-2}(x) = x^2-2$, $T_{-1}(x)=x$, $T_0(x) = 2$, $T_1(x)=x$, $T_2(x) = x^2-2$, $T_3(x) = x^3-3x$,
\item[] $U_{-2}(x) = -1$, $U_{-1}(x) = 0$, $U_0(x) = 1$, $U_1(x) = x$, $U_2(x) = x^2-1$, $U_3(x) = x^3-2x$.
\end{itemize}
\end{example}

\section{A result about $g(p)$}

\begin{lemma}
\label{epsilonf}
Let $p \in \mathbb N$ be a prime number not dividing the discriminant $\Delta$ of $K$. Then, with the usual notation for the Jacobi symbol, we have: \eq{g(p) \ \vert \ p-\legendre{\Delta}{p}.}
\end{lemma}
\begin{proof}
Recall $\ep = \tfrac{a + b\sqrt{d} } 2$, with $a,b \in \mathbb N$, and set $l = p-\legendre{\Delta}{p}$. When we write $e\equiv f$, we always mean $e \equiv f \text{ mod $p$}$. \\ \\
Let us start with the case $(p)$ is non-split, i.e., in particular $\legendre{\Delta}{p} = -1$. We need to show that \eq{p \ | \ 2^{p+1}(\ep^{p+1}-1).} As \eq{\binom{p}{k} \equiv \left\{ \begin{array}{c} 0 \text{ , if $0 < k < p$} \\  \  1 \text{ , if $k \in \{0,p\}$}  \end{array}  \right. ,} we have \eq{2^{p+1} \ep^{p+1} = (a+b\sqrt{d} )^p(a+b\sqrt{d} ) = a^{p+1}+b^{p+1} d^{(p+1)/2} +(a^p b + a b^p d^{(p-1)/2} ) \sqrt{d}.} From the assumptions, we know \eq{d^{(p+1)/2} \equiv -d \text{ and } d^{(p-1)/2} \equiv -1 .} Moreover, using Fermat's little theorem, we see \eq{a^{p+1}+b^{p+1} d^{(p+1)/2} \equiv a^2 -d b^2 \equiv 4} and \eq{a^p b + a b^p d^{(p-1)/2} \equiv ab - ab = 0.} This gives \eq{2^{p+1}\ep^{p+1} \equiv 4 } and thus $p \ \vert \ 2^{p+1}(\ep^{p+1}-1)$. \\ \\
Now we look at the case where $(p)$ splits. In this case we need to show that \eq{p \ \vert \ (\ep^{l}-1) (\ep^{-l}-1) = 2 - (\ep^{l} + \ep^{-l}) = 2- T_l(a).} We calculate \eq{2^l \, T_l(a) = 2 \sum_{k=0}^{l/2} \binom{l}{2k}a^{2k} (b^2d)^{l/2-k}.} If $\legendre{\Delta}{p} = 1$, there exists $1\leq r \leq l/2$ such that $r^2 \equiv d$. Using $\binom{p-1}{2k} \equiv 1 $ for $0 \leq k \leq l/2$ and Fermat's little theorem, we obtain \eq{2^l \, T_l(a) \equiv T_l(a) \equiv 2 (br)^l \sum_{k=0}^{l / 2}(a^2b^{-2}d^{-1})^k \equiv 2 \tfrac{(a^2b^{-2}d^{-1})^{l+1}-1}{(a^2b^{-2}d^{-1})-1} \equiv 2 \tfrac{(a^2b^{-2}d^{-1})-1}{(a^2b^{-2}d^{-1})-1} = 2.} If $\legendre{\Delta}{p}=-1$, we have again $d^{l/2}\equiv -d $. From \eq{\binom{p+1}{k} \equiv \left\{ \begin{array}{c} 0 \text{ ,if $2 \leq k \leq p-1$} \\  \  1 \text{, if $k \in \{0,1,p,p+1\}$}  \end{array}  \right. ,} together with Fermat's little theorem, we get \eq{2^l \, T_l(a) \equiv 2(a^l+b^l d^{l/2}) \equiv 2(a^2 - d b^2 ) = 8 .} Thus, we get $p \ \vert \ T_l(a) - 2$, finishing the proof.
\end{proof}

\bibliographystyle{plain}
\bibliography{shintani_values}

\end{document}